\documentclass[a4paper,11pt,dvipsnames]{article}

\usepackage[T1]{fontenc}
\usepackage[utf8]{inputenc}

\usepackage[english]{babel}

\usepackage{amssymb,amsthm,mathtools}
\usepackage{amsfonts,dsfont}
\usepackage[mathcal]{eucal}

\usepackage{lmodern}

\usepackage[text={13.5cm,21cm},centering,paperwidth=18.5cm,paperheight=26cm]{geometry}
\usepackage{csquotes,microtype,xcolor}
\usepackage{graphicx,siunitx,setspace}
\usepackage{algorithm,algpseudocode}

\usepackage{caption,subcaption}
\usepackage{natbib}
\usepackage[unicode,colorlinks,bookmarksnumbered]{hyperref}
\hypersetup{linkcolor=blue,citecolor=blue,urlcolor=blue}

\usepackage{orcidlink}

\usepackage[capitalise]{cleveref}
\crefname{section}{Sect.}{Sects.}
\crefname{figure}{Fig.}{Figs.}

\theoremstyle{plain}
\newtheorem{theorem}{Theorem}[section]
\newtheorem{hypothesis}{Hypothesis}
\newtheorem{proposition}[theorem]{Proposition}
\newtheorem{corollary}[theorem]{Corollary}

\theoremstyle{definition}
\newtheorem{definition}[theorem]{Definition}

\newtheorem{model}{Model}

\newcommand*{\N}{\ensuremath \mathbb{N}} 
\newcommand*{\R}{\ensuremath \mathbb{R}} 
\renewcommand*{\le}{\ensuremath \leqslant} 
\renewcommand*{\ge}{\ensuremath \geqslant} 
\newcommand*{\seg}[1]{\ensuremath \{1,\dots,#1\}} 
\newcommand*{\intd}[1]{\ensuremath \mathrm{d}#1} 
\newcommand*{\Esp}{\ensuremath \mathbb{E}} 
\newcommand*{\Prob}{\ensuremath \mathbb{P}} 
\newcommand*{\Exp}{\ensuremath \mathcal{E}} 
\newcommand*{\majsto}{\ensuremath \le_{\mathrm{st}}} 
\newcommand{\Break}{\State \textbf{break}} 
\newcommand*{\weightnorm}[1]{\ensuremath \Vert #1 \Vert_{\varepsilon}} 

\DeclareSIUnit\mRNA{mRNA}
\DeclareSIUnit\prot{prot}

\newcommand*{\degradop}{\ensuremath \mathcal{D}} 
\newcommand*{\creatop}{\ensuremath \mathcal{C}} 
\newcommand*{\burstop}{\ensuremath \mathcal{B}} 
\newcommand*{\commonjumprate}[1]{\ensuremath \lambda_{#1}^{\mathrm{1,2}}} 
\newcommand*{\jumpratecone}[1]{\ensuremath \lambda^1_{#1}} 
\newcommand*{\jumpratectwo}[1]{\ensuremath \lambda^2_{#1}} 

\newcommand{\ARNMP}{\ensuremath (Y(t))_{t \ge 0}} 

\newcommand*{\Comp}{\ensuremath (\CompL(t))_{t \ge 0}} 
\newcommand*{\CompL}{\ensuremath U} 
\newcommand*{\CompSmL}{\ensuremath u} 

\newcommand{\biokoni}{\ensuremath k_{\text{on},i}}
\newcommand*{\trueinteraction}{\ensuremath \sigma} 
\newcommand{\koni}{\ensuremath \alpha_i} 
\newcommand{\kon}{\ensuremath \alpha} 

\newcommand*{\rescalefunc}{\ensuremath \varphi} 
\newcommand*{\lipnorm}{\ensuremath \ell} 

\newcommand{\mindun}{\ensuremath d_1} 

\newcommand{\epsi}[1]{\ensuremath \varepsilon_{#1}} 
\newcommand*{\roverdun}{\ensuremath \kappa} 

\newcommand*{\fctgeneh}{\ensuremath G} 

\begin{document}

\title{\vspace{-18mm}\textbf{Quantitative ergodicity for gene regulatory networks with transcriptional bursting}}

\author{\normalsize Mathilde Gaillard\,\orcidlink{0009-0002-9520-3814} and Ulysse Herbach\,\orcidlink{0000-0002-0972-385X}}

\date{\normalsize Université de Lorraine, CNRS, Inria, IECL, F-54000 Nancy, France\\\small\texttt{\{mathilde.gaillard, ulysse.herbach\}@inria.fr}}

\maketitle

\begin{center}
\begin{minipage}{12.25cm}
\small\noindent\textbf{Abstract.}
We study the long-term behavior of two piecewise-deterministic Markov processes used to model stochastic gene regulatory networks with bursting dynamics. Under regularity assumptions on the jump rate, we prove the existence and uniqueness of the stationary distribution for an arbitrary number of interacting genes and an arbitrary strength of interaction. Using coupling methods, we also provide explicit upper bounds for the convergence to equilibrium in terms of Wasserstein distances.

\vspace{3.2mm}

\noindent\textbf{Keywords:} Gene regulatory networks, Stochastic gene expression, Ergodicity, Piecewise-deterministic Markov processes, Coupling, Wasserstein distance
\end{minipage}\medskip
\end{center}


\section{Introduction}

Piecewise-deterministic Markov processes (PDMPs) belong to a general class of stochastic models such that the randomness comes only from a jump mechanism. Several lines of work have studied their long-time behavior and characterized their ergodicity properties through the ergodicity of some particular discrete-time Markov chains (e.g., see~\citet{Costa2008} for a completely general PDMP and~\citet{Benaim2015} for slightly more specific models). If these results may be used to establish the existence and uniqueness of a stationary distribution, they fail to provide quantitative rates of convergence to the equilibrium. For convergence rates in total variation distance, a Foster-Lyapounov criterion~\citep{Meyn2009} can be used to guaranty the existence of a particular (but not explicit) rate (see for example~\citet{Bierkens2019}). To obtain convergence in the Wasserstein distance sense, coupling methods are commonly used, but explicit rates derived so far seem to have been restricted either to a subclass of PDMPs, the class of switched dynamical systems~\citep{Benaim2012}, or to some other very specific models~\citep{Malrieu2015,Chafai2010,Bardet2013}. The former refers to PDMPs with a state space \( \R^n \times E \) where \( E \) is a finite set, with the continuous component following a vector field that switches  when the discrete component jumps. The latter fall into a category of PDMPs with only one component, driven by a single flow and directly impacted by jumps. So far, for this last category of models, proofs of ergodicity obtaining convergence rates rely on constant jump rates~\citep{Malrieu2015} or analytical expression of the jump rate~\citep{Bardet2013,Chafai2010}. 

In this paper, we extend the method developed in~\citep{Benaim2012} for switched dynamical systems to some multivariate PDMPs with a single flow and unbounded trajectories. The resulting convergence rates are fully explicit. This methodology applies to non-constant jump rates and does not require an analytical expression. We focus here on two specific models formulated as PDMPs coming from molecular biology~\citep{Gaillard2025} and describing an arbitrary number of genes in interaction at the single-cell level. More precisely, in the following, we consider a network of \( n \) genes, possibly in interaction, where each gene \( i \in \seg{n} \) is described by its mRNA level \( M_i \) and its protein level \( P_i \). The dynamics of each quantity is given by
\begin{equation}\label{eq : biological MP model}
\left\{\begin{array}{rl}
M_i(t) &\xrightarrow{\biokoni} M_i(t) + \Exp \left( b_i \right) \vspace{1.5mm} \\
{M_i}'(t) &= -d_{0,i} M_i(t) \vspace{1.5mm} \\
{P_i}'(t) &= s_{1,i} M_i(t) - d_{1,i} P_i(t)
\end{array}\right.
\end{equation}
where \( \Exp(b_i) \) denotes the exponential distribution of parameter \( b_i \) and \( \biokoni \) is a function of proteins levels \( P(t) = \left( P_1(t), \dots, P_n(t) \right)\), encoding possible interactions between genes. The first item in \eqref{eq : biological MP model} means that with rate \( \biokoni \), the mRNA quantity jumps from its current value to a new (random) one, chosen by adding an exponential random variable to the current value. This reproduces an observed biological phenomenon called transcriptional bursting : mRNA copies are only produced, many at a time, in very short periods~\citep{Suter2011}. At the limit, these periods are modeled by instantaneous jumps called \emph{bursts}. From a more theoretical point of view, this model appears as a non-degenerate limit of the two-state model considered in~\citep{Herbach2017,Pajaro2017} in a specific parameters regime (see~\citep{Gaillard2025} for more details). It also emerged as the limit of continous-time Markov chains modelling chemical reactions with different time scales~\citep{Crudu2012, Chen2019}. 

While this model taking into account both mRNA and proteins is biologically relevant (in particular in reproducing scRNA-seq data~\citep{Ventre2023}), it remains mathematically complex. Regarding theoretical results, the following simplified version, where the mRNA layer is removed and bursts applied directly to proteins, is often considered 
\begin{equation} \label{eq : biological P model}
\left\{\begin{array}{rl}
\widehat{P}_i(t) &\xrightarrow{\biokoni} \widehat{P}_i(t) + \Exp(c_i) \vspace{1.5mm} \\
{\widehat{P}_i}{}'(t) &= -d_{1,i} \widehat{P}_i(t)
\end{array}\right.
\end{equation}
This simplified model is relevant when the degradation rate of mRNA is large relative to the degradation rate of proteins, namely in the regime where for all \( i \in \seg{n} \) \( d_{0,i} \gg d_{1,i} \). While this is not necessary for those models to be well defined, it is the regime in which the difference between proteins deterministic flows is small. To maintain a small difference when there is a burst, the optimal choice is to set 
\begin{equation} \label{eq_:_def_epsilon}
c_i = \frac{b_i d_{1,i}}{\epsi{i} s_{1,i}}, \quad \text{with} \quad \epsi{i} = \frac{d_{1,i}}{d_{0,i} - d_{1,i}}.
\end{equation}
Note that this can be done without loss of generality, since multiplying \( c_i \) by a factor is equivalent to dividing the proteins by the same factor and modifying the interaction function. Physical parameters used in \eqref{eq : biological MP model} and \eqref{eq : biological P model} are summarized in \cref{table : parameters}.

\begin{table}[ht]
\caption{Summary of the parameters considered in biological models \eqref{eq : biological MP model} and \eqref{eq : biological P model} for a fixed gene \( i \). The first column shows the notation used, the second the physical dimensions of the parameter, and the third its physical interpretation.}
\centering
\begin{tabular}{clr}
\hline
Notation & Units & Physical interpretation \\
\hline
\( d_{0,i} \) & \unit{\per\second} & mRNA degradation rate\\
\( d_{1,i} \) & \unit{\per\second} & protein degradation rate  \\
\( s_{1,i} \) & \unit{\prot}.\unit{\per\second}.\unit{\per\mRNA} & translation rate \\
\( 1 \slash b_i \) & \unit{mRNA} & mRNA burst size \\
\( 1  \slash c_i \) & \unit{\prot} & protein burst size \\
\hline
\end{tabular}
\label{table : parameters}
\end{table}

The aim of this paper is to show the ergodicity and provide convergence rates for these two models. Although proofs are made explicitly for these two specific PDMPs, the technique can be transposed to a wider class of bursty PDMPs with a few dynamic restrictions. The technique used here is an extension of the one developed in~\citep{Benaim2012}. As a result, \cref{thm : bound over wasserstein distance P-P} extends Theorem 3.4 in~\citep{Chen2019}. 

This paper is organized as follows. In \cref{section : notation and main results} we present the mathematical framework and the results. The \cref{section : discussion} is dedicated to discussion and numerical illustrations on the results. A conclusion and some perspectives for future work are given in \cref{section_:_conclusion_and_perspectives}. The proofs of results in \cref{section : notation and main results} can be found in \cref{section : Proofs Couplings P-P and MP-MP}.

\section{Notation and main results} \label{section : notation and main results}

This section is devoted to the main mathematical results. We first give some reminder of needed mathematical tools, then in \cref{subsection : mathematical framework} we introduce the mathematical framework. In \cref{subsection : quantitative ergodicity} we state the main theorems which are ergodicity results along with fully explicit convergence rates. 

\bigskip
 
In all this paper we consider only random variables \( X \) with a finite first moment, i.e.
\[ \Esp[ \vert X \vert ] < \infty. \]
We consider also the Wasserstein metric of order 1 on \( (\R^n, \Vert \cdot \Vert) \), where for all \( x = (x_1, \dots, x_n) \in \R^n \)
\[ \Vert x \Vert = \Vert x \Vert_1 = \sum_{i=1}^n \vert x_i \vert. \]
The Wasserstein metric between \( \mu \) and \( \nu \) two probability distributions on \( \R^n \) is defined by
\[ W_1 (\mu, \nu) = \inf \left\{ \Esp[ \Vert X - Y \Vert ] ; X \sim \mu, Y \sim \nu \right\}, \] 
where the infimum is taken over all couplings, namely every join distribution for \( (X,Y) \) with the right marginals. In our setting, another usefull norm is the weighted norm defined for all \( x = (x_1, \dots, x_n) \in \R^n \) by
\[ \weightnorm{x} = \sum_{i=1}^n \epsi{i} \vert x_i \vert, \]
where \( \epsi{1}, \dots, \epsi{n} \) are defined at \eqref{eq_:_def_epsilon}.

We consider the element-wise order on \( \R^n \), we write \( x \le y \) with \( x = (x_1,\dots, x_n) \) and \( y = (y_1, \dots, y_n) \) if for all \( i \in \seg{n} \), \( x_i \le y_i \). A function \( f : \R^n \longrightarrow \R \) is said to be increasing if 
\[ \forall x, y \in \R^n, \quad x \le y \Longrightarrow f(x) \le f(y). \]

We make also heavy use of the stochastic ordering which is a partial order on probability distributions. We write \( \mu \majsto \nu \) if
\[ \int f \intd{\mu} \le \int f \intd{\nu} \]
for every function \( f : \R^n \longrightarrow \R \) monotone increasing. 
 In this case we say that \( \nu \) is \emph{stochastically larger} than \( \mu \). 
To sometimes simplify the notation, we write it with random variables.
For example, for two random variables \( X \) and \( Y \), we write
\[ X \majsto Y, \]
if the distribution of \( Y \) is stochastically larger than the distribution of \( X \).
We intensively use the two following characterizations of the stochastic ordering. First, a stochastic inequality can be transformed into an inequality almost surely between random variables. Indeed, \( X \) is stochastically larger than \( Y \) if and only if there exists a coupling of \( (X, Y) \) such that \( X \le Y \) almost surely. In the following, if two distributions verify a stochastic inequality, we will often implicitly consider such a coupling for the associated random variables.
Secondly, we sometimes use that when \( X \) and \( Y \) are random variables on \( \R \), with cumulative distribution functions \( F_X \) and \( F_Y \), 
\[ X \majsto Y \Longleftrightarrow \forall t \in \R \quad F_X(t) \ge F_Y(t). \]

\subsection{Mathematical framework} \label{subsection : mathematical framework}

In order to simplify notation, we introduce a normalized version of the quantities defined at \eqref{eq : biological MP model} and \eqref{eq : biological P model}. For every \( i \in \seg{n} \), we set
\begin{align*}
X_{i}(t) &= \frac{d_{1,i} b_i}{s_{1,i} \epsi{i}} \widehat{P}_{i}(t), \\
Y_{i}(t) &= \frac{b_i}{\epsi{i}} M_{i}(t), \\
Z_{i}(t) &= \frac{d_{1,i} b_i}{s_{1,i} \epsi{i}} P_{i}(t),
\end{align*}
where \( \epsi{i} \) is defined at \eqref{eq_:_def_epsilon} and other parameters comes from \eqref{eq : biological MP model} and \eqref{eq : biological P model}.
We also set 
\[ \koni = \biokoni \circ \rescalefunc, \]
where, \( \rescalefunc \) is defined for \( x = (x_1, \dots, x_n) \in \R^n \) by
\[
\rescalefunc(x) = \left( \dfrac{s_{1,1} \epsi{1}}{d_{1,1} b_1} x_1, \dots , \dfrac{s_{1,n} \epsi{n}}{d_{1,n} b_n} x_n \right).
\]
Models corresponding to these new dimensionless quantities are as follow.
\begin{model}[Reduced protein-only model]\label{model : Simplified Mathematical Model}
\[
\left\{\begin{array}{rl}
X_i(t) &\xrightarrow{\koni(X(t))} X_i(t) + \Exp(1) \vspace{1.5mm} \\
{X_i}'(t) & = -d_{1,i} X_i(t)
\end{array}\right.
\]
\end{model}
\begin{model}[Complete mRNA-protein model]\label{model : Mathematical Bursty Model}
\[ 
\left\{\begin{array}{rl}
Y_i(t) &\xrightarrow{\koni(Z(t))} Y_{i}(t) + \Exp \big( \epsi{i} \big) \vspace{1.5mm} \\
{Y_i}'(t) &= -d_{0,i} Y_i(t) \vspace{1.5mm} \\
{Z_i}'(t) &= d_{1,i} (Y_i(t) - Z_i(t))
\end{array}\right.
\]
\end{model}
These forms highlight the fact that \( s_{1,i} \), \( b_i \) and \( c_i \) are just scaling constants. Remember that \( c_i \) does not appear in the normalization since we used the relation \eqref{eq_:_def_epsilon}. As their respective biological counterparts (\eqref{eq : biological MP model} and \eqref{eq : biological P model}), \cref{model : Mathematical Bursty Model} and \cref{model : Simplified Mathematical Model} are PDMPs. This means that the motion of these processes is characterized by a flow, namely a deterministic motion between jumps, a jump rate to control when jumps occur, and a transition kernel to draw a new position. In the simple case of one-dimensional \cref{model : Simplified Mathematical Model}, starting from \( x \) and until the instant of the first jump \( T \), we have
\[ X(t) = xe^{-d_1 t}. \]
The distribution of \( T \) is characterized by its survival function
\[
\Prob_{x} (T > t ) = \exp \left( - \int_0^t  \kon( xe^{-d_1 s}) \intd{s} \right).
\]
Finally, at time \( T \), the new position is drawn according to the kernel \( Q \) defined by
\[ Q(xe^{-d_1 T}, \cdot) = \delta_{xe^{-d_1 T}} + \Exp(1), \]
where \( \delta_y \) denotes the Dirac mass at \( y \). The motion restarts from this new point as before. PDMPs such as the ones considered in this paper, can be characterized by their infinitesimal generators (see~\citep{Malrieu2015, Benaim2015} for a more general overview). Since in this paper, we consider only PDMPs with the same three types of dynamics : degradation, creation and burst, we introduce the following notation. For every \( x = (x_1, \dots, x_d) \in \R^d \) and \( f \) regular enough, we define the degradation operator
\[ \degradop_{x_i} f(x) = - x_i \partial_{x_i} f(x), \]
the creation operator
\[ \creatop_{x_i}^{x_j} f(x) = x_j \partial_{x_i} f(x), \]
and finally the burst operator defined for a subset of at most two variables \( \{ x_i, x_j \} \) and some associated parameters \( \{ h_i, h_j \} \) by
\[ \burstop_{x_i, x_j}^{h_i, h_j} f(x) = \int_0^{+ \infty} (f(x + h_i s e_i + h_j s e_j) - f(x) ) e^{-s} \intd{s} \]
where \( e_1, \dots, e_d \) is the canonical basis of \( \R^d \). We may omit the parameters \( \{ h_i, h_j \} \) if they are equal to one and there is no ambiguity. To set an example, the infinitesimal generator of \cref{model : Simplified Mathematical Model} fully written is
\[
L_{\text{P}} f(x) = \sum_{i=1}^n \left( -d_{1,i} x_i \partial_{x_i} f(x) + \koni(x) \int_0^{+ \infty} (f(x + s e_i) - f(x)) e^{- s} \intd{s} \right),
\]
but we rather use the simplify version with our notations
\begin{equation}\label{eq : generator P n genes}
L_{\text{P}} = \sum_{i=1}^n d_{1,i} \degradop_{x_i} + \koni \burstop_{x_i}. 
\end{equation}
Similarly, the generator of \cref{model : Mathematical Bursty Model} is
\begin{equation}\label{eq : Générateur MP n gènes}
L_{\text{MP}} = \sum_{i=1}^n d_{1,i} \left( \degradop_{z_i} + \creatop_{z_i}^{y_i} \right) + d_{0,i} \degradop_{y_i} + \koni \burstop_{y_i}^{\epsi{i}^{-1}},
\end{equation}
where \( \koni \) is a function of vector \( z \) only. Such models will be called \emph{bursty} PDMPs, in reference to the transcriptional bursting phenomenon mentioned in the introduction. More broadly, we refer to PDMPs with only real components driven by a unique deterministic flow that is directly impacted by jumps. These jumps will also be referred to as \emph{bursts}.

\bigskip

In all the remaining of this paper, we will place ourselves under the following hypothesis.
\begin{hypothesis}\label{hyp : kon bounded and Lipschitz}
We assume that 
\[ 
\kon : \left|\begin{array}{lll}
\R^n & \longrightarrow & \R^n \\
x & \longmapsto & \left( \kon_1(x), \dots, \kon_n(x) \right)
\end{array}\right.
\]
is bounded and Lipschitz continuous.
\end{hypothesis}
More precisely, we assume that for each gene \( i \), \( \koni \) can be decomposed as
\[
\koni(x) = k_{0,i} + (k_{1,i} - k_{0,i}) (\trueinteraction_i \circ \rescalefunc) (x),
\]
where \( k_{0,i} \) and \( k_{1,i} \) are respectively the infimum and the supremum of \( \biokoni \), and \( \trueinteraction_i \circ \rescalefunc : \R^n \longrightarrow [0,1] \) is a \( \ell_i \)-Lipschitz continuous function. Thus using \cref{hyp : kon bounded and Lipschitz}, we immediately have that
\begin{equation}\label{eq:consequence_of_the_bounded_Lipschitz_assumption}
\Vert \kon(x) - \kon(z) \Vert \le r \left( 1 \wedge \lipnorm \Vert x - z \Vert \right), 
\end{equation}
where
\[ r = \sum_{i=1}^n k_{1,i} - k_{0,i} \quad \text{and} \quad \lipnorm = \frac{1}{r} \sum_{i=1}^n \ell_i. \]
The constant \( \lipnorm \) describes the aggregated strength of the interaction. Indeed, every Lipschitz constant \( \ell_i \) measures the extent to which a difference in protein levels of gene \( i \) may results in an difference in jump rate. From a mathematical point of view, the Lipschitz assumption is sufficient to ensure that PDMPs considered here are well defined, namely non explosive (see~\citet{Chen2019} for the proof). Nonetheless, it is biologically reasonable to assume that \( \kon \) is bounded. 

The second hypothesis that we will consider throughout this paper is the following.
\begin{hypothesis}\label{hyp : degradation rates}
For every gene \( i \in \seg{n} \), the degradation rate of mRNA molecules is greater than the degradation rate of proteins, that is :
\[ d_{0,i} > d_{1,i}. \]
\end{hypothesis}
This hypothesis comes from biological experiences showing that proteins are more stable than mRNA molecules~\citep{Schwanhausser2011}. Our proof scheme, strongly relies on this hypothesis. In this setting, the solutions of the ODEs systems describing the deterministic dynamics between jumps, are given by
\begin{equation}\label{eq_:_sol_analytique_EDO_modele_P}
X_{i}(t) = X_{i}(0) e^{-d_{1,i}t}
\end{equation}
for \cref{model : Simplified Mathematical Model}, and by
\begin{equation}\label{eq_:_sol_analytique_EDO_modele_MP}
\left\{ \begin{array}{ll}
Y_{i}(t) & = Y_{i}(0) e^{-d_{0,i}t} \\
Z_{i}(t) & = Z_{i}(0) e^{-d_{1,i}t} + \epsi{i} Y_{i}(0) \left( e^{-d_{1,i}t} - e^{-d_{0,i}t} \right)
\end{array} \right.
\end{equation}
for \cref{model : Mathematical Bursty Model}.

\subsection{Quantitative ergodicity} \label{subsection : quantitative ergodicity}

In this section we present ergodicity results for both \cref{model : Simplified Mathematical Model} and \cref{model : Mathematical Bursty Model} along with their respective explicit rates of convergence. The proofs of the results in this section can be found in \cref{section : Proofs Couplings P-P and MP-MP}.

\bigskip

In the following of this section we will need the next notation.
\begin{definition}\label{def_:_notation_measures}
For \( i \in \{\emptyset, 1, 2\} \) we denote by \( \mu^i(0) \) the distribution of \( X^i(0) \), and by \( \mu^i(t) \) the distribution of \( X^i(t) \) starting from \( \mu^i(0) \) with the dynamics given  by \eqref{eq : generator P n genes}. Similarly, we also denote by \( \nu^i(0) \) (resp. \( \eta^i(0) \)) the distribution of \( Z^i(0) \) (resp. \( Y^i(0) \)). Finally, \( \nu^i(t) \) (resp. \( \eta^i(t) \)) stands for the distribution of \( Z^i(t) \) (resp. \( Y^i(t) \)) starting from \( \nu^i(0) \) (resp. \( \eta^i(0) \)), the whole dynamics of \( (Y^i(t), Z^i(t)) \) being given by \eqref{eq : Générateur MP n gènes}.
\end{definition}
Before stating the main results of this paper, we define for \( u \in \R_+ \), the function
\begin{equation} \label{eq_:_proba_arret_des_sauts_processus_compagnon}
p^*(u) = e^{- \roverdun (1 \wedge \lipnorm (u \vee \roverdun))} \left( \frac{1}{1 \vee \lipnorm (u \vee \roverdun)} \right)^{\roverdun}
\end{equation}
where
\begin{equation}
\roverdun = \frac{r}{\mindun}, \quad \mindun = \min_{1 \le i \le n} d_{1,i}.
\end{equation}
and where for every \(x, y \in \R \)
\[ x \wedge y = \min (x, y), \quad  x \vee y  = \max(x,y). \]
It is easy to see that for every \( u \in \R_+ \), \( p^*(u) \) defines a probability. Roughly speaking, given a value of \( \roverdun \), a value of \( \lipnorm \) and an initial expectation \( u \), \( p^*(u) \) returns the probability of getting (at some random times), the perfect synchronization of coordinates in the couplings constructed in \cref{section : Proofs Couplings P-P and MP-MP}. However, since the couplings will be introduced later on, a value \( p^*(u) \) must be interpreted as a measure of synchronization between \( \mu^1(t) \) and \( \mu^2(t) \) (or \( \nu^1(t) \) and \( \nu^2(t) \)). The smaller \( p^*(u) \) is, the longer it takes for \( \mu^1(t) \) and \( \mu^2(t) \) to be close (in the Wasserstein distance sense). It is clear that \( p^* \) is a non-increasing function of \( u \), as well as a non-increasing function of \( \roverdun \) and \( \lipnorm \). This is illustrated in \cref{fig_:_p_star}. Intuitively, a large initial distance between probability measures slows down synchronization as well as strong interaction between genes. Indeed, as we already mentioned, \( \lipnorm \) is one measurement of the strength of interaction between genes, but it appears here that it is also the case for \( \roverdun \). When \( \roverdun \) is large, either \( r \) is large and bursts are frequent, or \( \mindun \) is small and protein are degraded slowly, and both cases result in larger synchronization time. An other useful quantity regarding the strength of interaction is 
\begin{equation}\label{eq_:_tau}
\tau = r \wedge \mindun,
\end{equation}
that is the value of the slowest dynamic rate, either burst frequency or protein degradation.

\begin{figure}[ht]
\centering\includegraphics[width=\textwidth]{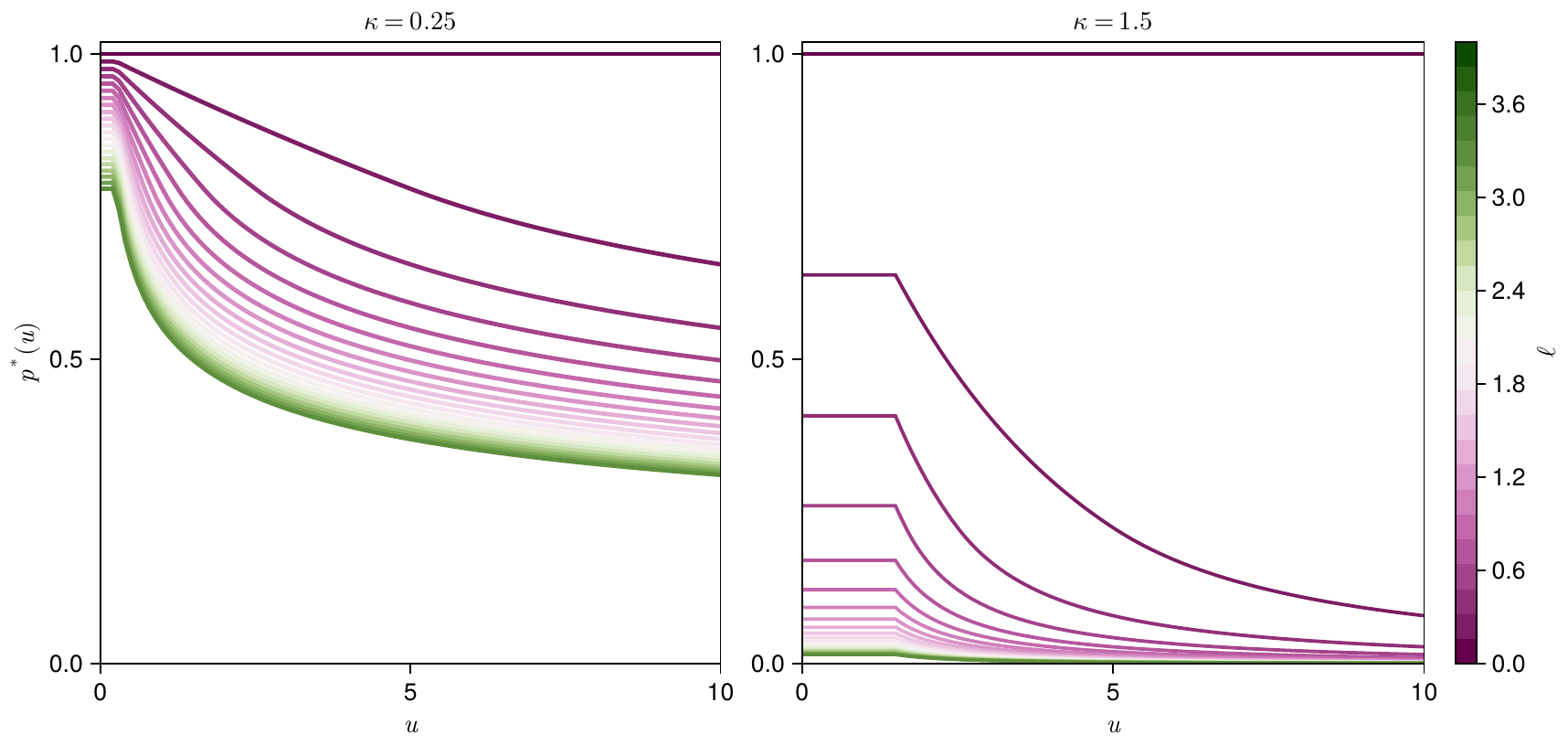}
\caption{Illustration of the function \( p^* \) defined at \eqref{eq_:_proba_arret_des_sauts_processus_compagnon} for different parameter values.}
\label{fig_:_p_star}
\end{figure}

We are now able to state our result in the case of \cref{model : Simplified Mathematical Model}.

\begin{theorem} \label{thm : bound over wasserstein distance P-P}
Given \( \mu^1(0) \) and \( \mu^2(0) \) and set 
\[ W_1(0) = W_1(\mu^1(0), \mu^2(0)) \quad \text{and} \quad p^* = p^*(W_1(0)). \] 
Then for all \( t \ge 0 \),
\[ W_1( \mu^1(t), \mu^2(t)) \le \left( W_1(0) \vee \kappa  +  \gamma e^{-1}(W_1(0) + \tau (1 - p^*)) t \right) e^{ - \gamma t} \]
where
\[ \gamma = \frac{p^* \tau \mindun}{p^* \tau + \mindun}, \]
and \( \tau \) is defined at \eqref{eq_:_tau}.
\end{theorem}

We have a similar result for \cref{model : Mathematical Bursty Model}, but with a \emph{different} initial condition. 

\begin{theorem} \label{thm : bound over wasserstein distance MP-MP}
Given \( \nu^i(0) \) and \( \eta^i(0) \) for every \( i \in \{1, 2\} \) and set 
\[ w_0 = W_1(\nu^1(0), \nu^2(0)) + \widetilde{W}_1(\eta^1(0), \eta^2(0)) \quad \text{and} \quad p^* = p^*(w_0), \] 
where \( \widetilde{W}_1 \) stands for the Wasserstein metric on \( (\R^n, \weightnorm{\cdot}) \)
Then for all \( t \ge 0 \),
\[ W_1( \nu^1(t), \nu^2(t)) \le \left( w_0 \vee \roverdun + \gamma e^{-1}(w_0 + \tau (1 - p^*)) t \right) e^{- \gamma t} \]
where 
\[ \gamma = \frac{p^* \tau \mindun}{p^* \tau + \mindun}, \]
and \( \tau \) is defined at \eqref{eq_:_tau}.
\end{theorem}

A classical result states that the set of probability measures with a finite first moment is a complete metric space for the Wasserstein distance \( W_1 \) (see \citet{Villani2009}). Thus we can deduce from the \cref{thm : bound over wasserstein distance P-P} and \cref{thm : bound over wasserstein distance MP-MP} that both \cref{model : Simplified Mathematical Model} and \cref{model : Mathematical Bursty Model} are ergodic.

\begin{corollary}
\cref{model : Simplified Mathematical Model} admits a unique stationary distribution denoted by \( \mu_{\infty} \). Moreover, with notation of \cref{thm : bound over wasserstein distance P-P} we have that for all \( t \ge 0 \)
\[ W_1( \mu(t), \mu_{\infty}) \le \left( W_1(0) \vee \kappa  +  \gamma e^{-1}(W_1(0) + \tau (1 - p^*)) t \right) e^{ - \gamma t}. \]
\end{corollary}

\begin{corollary}
The protein coordinate of \cref{model : Mathematical Bursty Model} admits a unique stationary distribution denoted \( \nu_{\infty} \). Moreover, with notation of \cref{thm : bound over wasserstein distance MP-MP} we have that for all \( t \ge 0 \)
\[ W_1( \nu(t), \nu_{\infty}) \le \left( w_0 \vee \roverdun + \gamma e^{-1}(w_0 + \tau (1 - p^*)) t \right) e^{- \gamma t}. \]
\end{corollary}

\section{Discussion} \label{section : discussion}

In this section, we discuss the results from \cref{section : notation and main results} and the mathematical interest of the proof's technique. Therefore, we recommend reading the beginning of \cref{section : Proofs Couplings P-P and MP-MP} beforehand.

In this paper we proved the ergodicity of two models of gene expression without any assumption on the Lipschitz constant \( \lipnorm \), namely the strength of interaction between genes. This is why \cref{thm : bound over wasserstein distance P-P} can be seen as an extension of~\citet[Theorem 3.4]{Chen2019} in which the ergodicity result needs the assumption (with our notation)
\begin{equation}\label{eq:dissipative_assumption_Chen2019}
\lipnorm < \roverdun^{-1}. 
\end{equation}
The setting in which we expressed our results is slightly different since we assume that the jump rate \( \kon \) is bounded. However, even if we intensively used this assumption to optimize the convergence speed \( \gamma \), our results stands without it. Removing the bounded assumption would, of course, change the expressions in both \cref{thm : bound over wasserstein distance P-P} and \cref{thm : bound over wasserstein distance MP-MP} and it would lead to a (slowest) exponential convergence rate. Nonetheless, the ergodicity and the exponential convergence speed remain true for \emph{any} value of the Lipschitz constant \( \lipnorm \). Note that, with our proof scheme, we recover the corresponding result of~\citet{Chen2019}. Indeed, without the bounded assumption, the jump rate of the companion process \( \Comp \) constructed in \cref{section : Proofs Couplings P-P and MP-MP} rewrites 
\[ \lambda_{\CompL}(u) = \sum_{i=1}^n \ell_i u, \]
where \( \ell_i \) are the ones defined in \cref{subsection : mathematical framework}. Using Grönwall's inequality as in \cref{prop_:_borne_uniforme_constante_esperance_U(t)}, we obtain that for all \( t \ge 0 \) 
\[ \Esp[ \CompL(t)] \le \Esp[ \CompL(0) ] \exp \left( \left( \sum_{i=1}^n \ell_i - \mindun \right) t \right), \]
which transcribed with notation in \cref{def_:_notation_measures} into 
\[ W_1( \mu^1(t), \mu^2(t)) = W_1(\mu^1(0), \mu^2(0)) \exp \left( \left( \sum_{i=1}^n \ell_i - \mindun \right) t \right). \]
The condition for this to converge is exactly \eqref{eq:dissipative_assumption_Chen2019}. This restriction remains strong as it constraints the strength of interactions between genes and prevents the model's interesting emerging properties from appearing. This is illustrated in \cref{fig:toggle_switch} for the toggle switch. This network of two genes that repress each other exhibits a bistable pattern, whereby one gene is \enquote{active} and the other \enquote{inactive}, with the roles alternating. Setting parameters such that \eqref{eq:dissipative_assumption_Chen2019} is satisfied restrict the dynamics and the pattern does not emerge. Instead each protein is expressed all the time.

\begin{figure}[ht]
\centering
\begin{subfigure}{\textwidth}
\centering
\includegraphics[width=\textwidth]{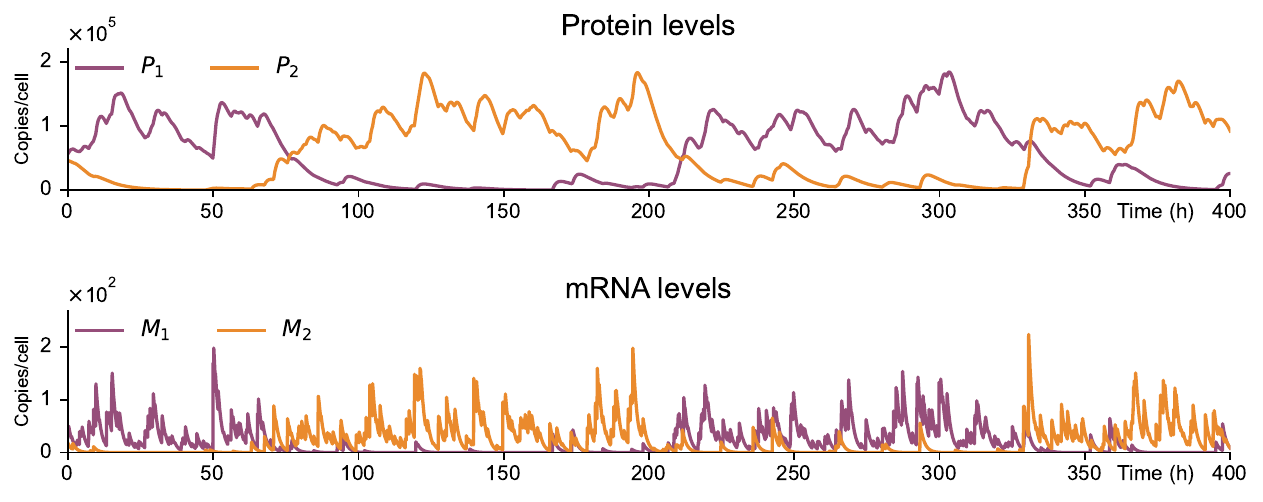}
\end{subfigure}

\bigskip
    
\begin{subfigure}{\textwidth}
\centering
\includegraphics[width=\textwidth]{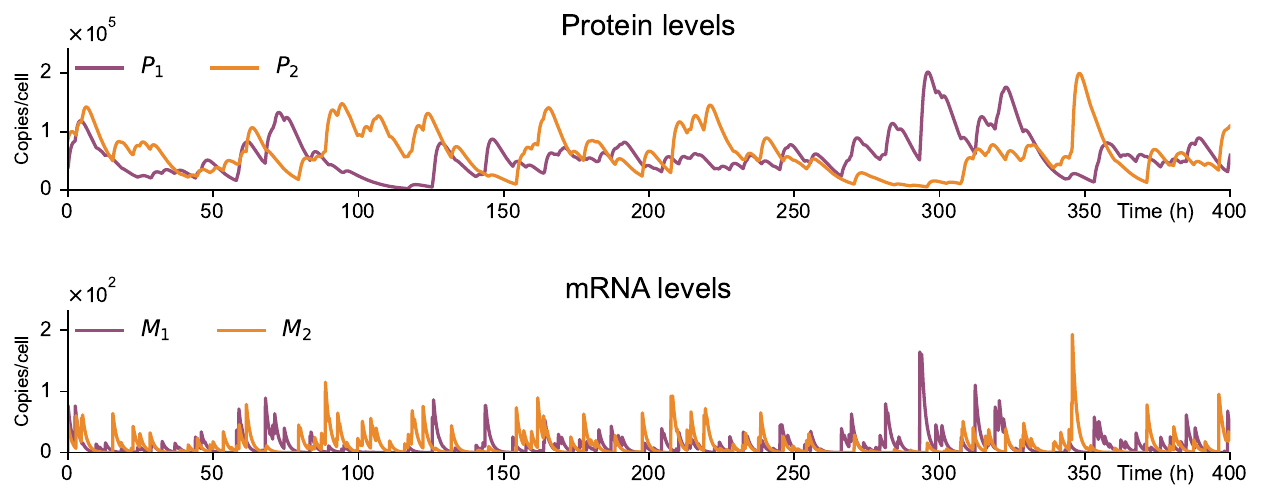}
\end{subfigure}
\small{\caption{Example of a trajectory of \cref{model : Mathematical Bursty Model} for a toggle switch (two genes repressing each other). The set of parameters used for the simulation resulting of the two top graphics gives \( \lipnorm < 6,6.10^{-1} \) and \( \roverdun^{-1} = 2,4.10^{-2} \). Thus \eqref{eq:dissipative_assumption_Chen2019} is not satisfied, however the bistable pattern of the model appears: the cell switches between two attractors where one of the two genes is \enquote{active} and the other is \enquote{repressed}. On the other hand, parameters for the simulation at the bottom satisfies \eqref{eq:dissipative_assumption_Chen2019} with \( \lipnorm < 4.1.10^{-3} \) and \( \roverdun^{-1} \approx 4,6.10^{-2} \) but interactions between genes are not strong to recover the bistable pattern.}}
\label{fig:toggle_switch}
\end{figure}

The proof's technique we use in this paper is in itself mathematically relevant. First, as discussed above, with the same coupling as in \citep{Chen2019} (defined at \eqref{eq : generator X1 X2 n genes}), it allows us to obtain a stronger result. Secondly, our method can be seen as an extension of the one developed~\citep{Benaim2012} to some unbounded settings. While the article~\citep{Benaim2012} considers the case of switching ODEs, our setting is different as we have bursty PDMPs, namely PDMPs driven by a unique flow and trajectories directly impacted by jumps. The mathematical technicalities resulting from this setting also make the proof interesting. More precisely in~\citep{Benaim2012}, when the discrete coordinates are different, the bounded assumption is used and the companion process stays constant, equal to the bound plus one (added to take into account unequaled discrete coordinates). When an event results in matching flows, namely equaled discrete components, the companion process jumps to the bound value without the plus one and then decreases exponentially fast to zero. Here, the companion process that naturally arises from our setting is a bursty PDMP similar to a one-dimensional version of \cref{model : Simplified Mathematical Model} with marginal generator
\[
L_{\CompL} = \mindun \degradop + \lambda_{\CompL} \burstop, \quad \text{where} \quad \lambda_{\CompL}(\CompSmL) = r( 1 \wedge \lipnorm \CompSmL).
\]
The first major technical difference is that its trajectories are unbounded and so we can not use a deterministic bound to almost surely control \( \CompL(t) \) for all time \( t \). The second difference concerns the control of the waiting times between jumps of \( \Comp \). Indeed, if as in~\citep{Benaim2012} the proof idea is to show that the companion process stops jumping and then use that it decreases to zero, in our setting, the instant of the last jump \( H \), can not be decomposed and controlled similarly. \( H \) is the sum of the (finite) waiting times of \( \Comp \), for which the distribution is only known conditionally to the post-jump value of the process. First, \cref{prop : maj sto finite waiting times} gives us a control over the distribution of finite waiting times, uniform with respect to the initial condition. Then, both differences explained above are handled with explicit construction of random variables using stochastic bounds. The mathematical detail of constructing all the random variables simultaneously is given in \cref{algo1}.  

Note that, unlike~\citep{Benaim2012}, we do not make the assumption that \( \kon \) is bounded by below by a strictly positive constant. This is an interesting property of PDMPs like \cref{model : Simplified Mathematical Model} and \cref{model : Mathematical Bursty Model}. While switching ODEs need this assumption to be ergodic, models considered here do not. The intuition is that switching ODEs needs switch to continuously happen in order to explore all the space, otherwise it may stay stuck in an attractor of a specific flow. For example, if the jump rate is equal to zero at some attractor point, the process may stop switching and therefore not be ergodic. For PDMPs considered here, if the jump rate is equal to zero at some attractor, the Dirac mass in the attractor value may be the stationary distribution of the process and the process can still be ergodic. This is exactly the case for the companion process \( \Comp \). It is easy to see that the distribution \( \delta_0 \) is stationary for this process, and the proof in \cref{section : Proofs Couplings P-P and MP-MP} shows that the process is ergodic as it almost surely has a finite number of jump and then tends to zero. Nonetheless, since biologically speaking we prefer that processes do not go extinct, a sufficient hypothesis for \cref{model : Simplified Mathematical Model} so that every protein coordinate keep bursting, would be that for every \( i \in \seg{n} \), \( \koni(0) \) is positive. With the same hypothesis, every coordinate of \( \ARNMP \) in \cref{model : Mathematical Bursty Model} keep bursting. 

While throughout this article, for biological relevance, the burst size has an exponential distribution, results in \cref{section : notation and main results} can be extend to other distribution by adapting the burst size distribution of \( \Comp \). We also highlight the fact that the \emph{same} process \( \Comp \) is used in both ergodicity proofs, where only the initial inequality that has to be verified by \( \CompL(0) \) differs. This suggests that such results can be obtained with the same technique for models with more layers, taking into account the different stages of mRNA maturation. The important criterion would be that proteins remain slower than other chemical species and drive the burst dynamics.

\section{Conclusion and perspectives}\label{section_:_conclusion_and_perspectives}

Under very few assumptions regarding the jump rate, we have proved that two commonly used models of gene expression are exponentially ergodic. While one result is new, the other extends the result of \citet{Chen2019} since it does not require an hypothesis on the strength of interactions. The method we used generalizes the technique developed by~\citet{Benaim2012}. It is intuitive to see that it can be applied to a wider class of bursty PDMPs respecting some dynamics conditions: the slowest chemical specie in the system should be the one regulating jumps. It could for example be used exactly the same way to obtain the ergodicity of models closer to the biological reality, taking into account more chemical species such as pre-mature mRNA \citep{Rudnicki2015}. This restriction comes from the fact that the companion process \( \Comp \), constructed in \cref{section : Proofs Couplings P-P and MP-MP}, is a one-dimensional process controlling the difference between two \(n\)-dimensional processes seen as two cells. This reduction of dimension imposes the use of the worst case scenario \eqref{eq:consequence_of_the_bounded_Lipschitz_assumption} to bound the jump rate of \( \Comp \). A more subtle upper bound would take into account that a protein may have a different impact according to its position in the network: the more a gene influences other, the longer it will take for two cells with different protein levels to synchronize. One way to tackle this problem may be to introduce a multidimensional companion process. However, such technique would require to add some knowledge about the jump rate \( \kon \). For example, without imposing a particular form to \( \kon \), only the topology of the gene regulatory network could already be an important information. A toy example would be a GRN structured as a directed tree, where every gene (a node) is influenced by at most one other gene (has at most one parent). A less naive companion process in this case may be a process with as many dimensions as the depth of the tree, each dimension handling the difference between proteins at identical depths. From the proof scheme it is seems intuitive that coordinates will successively stop jumping, but the mathematical difficulty has significantly increased. This perspective should be seen as an improvement for mathematical purposes, firstly because inferring GRNs is already in itself a very difficult task and an active field of research. Secondly, while the exponential ergodicity and the rates of these models may be interesting if one wants to simulate synthetic data according to the stationary distribution (for example before applying a stimulus \citep{Ventre2023}), to improve such rates seems not biologically relevant. Nonetheless, we think that the technique of the companion process should be further investigate in order to demonstrate the ergodicity of an even wider class of PDMPs. Increasing the dimension of the companion process may also be the way to extend the technique to other types of dynamic.

\section{Proofs of ergodicity} \label{section : Proofs Couplings P-P and MP-MP}

In this section we prove the ergodicity of both \cref{model : Simplified Mathematical Model} and \cref{model : Mathematical Bursty Model}. In both cases we construct a coupling of two processes with the same dynamics and bound the difference between the two trajectories by a specific Markov process, called the \emph{companion process}, whose construction is inspired by~\citep{Benaim2012}. This construction is first presented in \cref{subsection_:_processus_compagnon_P_modele} for \cref{model : Simplified Mathematical Model}. In \cref{subsection_:_etude_processus_compagnon} we demonstrate key properties of the companion process. Then in \cref{subsection_:_processus_compagnon_modele_MP} we construct the coupling for \cref{model : Mathematical Bursty Model}, using the same companion process to bound the difference between proteins trajectories. Finally in \cref{subsection_:_preuve_resultat_final_quantitative_ergodicity} we prove \cref{thm : bound over wasserstein distance P-P} and \cref{thm : bound over wasserstein distance MP-MP}. 

\bigskip

More specifically, the proof's structure is the following. Since the Wasserstein distance between two distributions is defined as an infimum over all couplings of these distributions, any coupling can be used to control the distance. For \cref{model : Simplified Mathematical Model} we will consider a Markov process \( (X^1(t), X^2(t))_{t \ge 0} \in (\R_+)^n \times (\R_+)^n \), where the dynamic of \( (X^1(t))_{t \ge 0} \) and \( (X^2(t))_{t \ge 0} \) are given by \eqref{eq : generator P n genes}. Similarly, for \cref{model : Mathematical Bursty Model}, we will consider \( (Y^1(t), Z^1(t), Y^2(t), Z^2(t))_{t \ge 0} \in (\R_+)^n \times (\R_+)^n \times (\R_+)^n \times (\R_+)^n \) , where the dynamics of \( (Y^1(t), Z^1(t))_{t \ge 0} \) and \( (Y^2(t), Z^2(t))_{t \ge 0} \) are prescribed by \eqref{eq : Générateur MP n gènes}. Given the definition of the Wasserstein distance, theses couplings should try to minimise
\[
\Esp[\Vert X^1(t) - X^2(t) \Vert] \quad \text{and} \quad \Esp[ \Vert Z^1(t) - Z^2(t) \Vert],
\]
for every \( t \ge 0 \).
In both cases the couplings \eqref{eq : generator X1 X2 n genes} and \eqref{eq_:_generator_MP_MP_U} are Markovian processes that can be seen as a joined distribution for two cells where the dynamics of each one is prescribed by one of the two models. Note that \eqref{eq : generator X1 X2 n genes} is the same coupling as in~\citep{Chen2019}. As the difference between protein levels, is not, in both cases, a Markovian process, we add a specific one-dimensional Markovian process \( \Comp \) to each coupling. This process \( \Comp  \in \R_+\) is called the \emph{companion process} and the idea of its name and construction comes from~\citep{Benaim2012}. As in~\citep{Benaim2012}, the remaining of the proof consists in showing that \( \Comp \) has a finite number of jumps and then control the exponential moment of \( H \), the instant of its last jump. Nevertheless, although the intuitions and the proof's structure are similar, the setting differs for two reasons. First the dynamics of our models, are driven by a unique flow. Secondly, the trajectories are unbounded because they are directly affected by jumps. This generates some mathematical technicalities that justify the mathematical interest in this proof. 

\subsection{A construction of the companion process} \label{subsection_:_processus_compagnon_P_modele}

Let us construct a coupling \( (X^{1}(t), X^{2}(t))_{t \ge 0} \in (\R_+)^n \times (\R_+)^n \) such that the marginals are driven by \eqref{eq : generator P n genes}. This results in the following infinitesimal generator:
\begin{equation}\label{eq : generator X1 X2 n genes}
L_{\text{P-P}} = \sum_{i=1}^n d_{1,i} \degradop_{x^{1}_i} + d_{1,i} \degradop_{x^{2}_i} + \jumpratecone{i} \burstop_{x^{1}_i} + \jumpratectwo{i} \burstop_{x^{2}_i} + \commonjumprate{i} \burstop_{x^1_i,x^2_i}
\end{equation}
with jump rates defined for every \( (x^1, x^2) \in (\R_+)^n \times (\R_+)^n \) by
\begin{align*}
\commonjumprate{i}(x^1,x^2) & = \koni(x^1) \wedge \koni(x^2), \\
\jumpratecone{i}(x^1,x^2) & = (\koni(x^{1}) - \koni(x^{2}) )^+, \\
\jumpratectwo{i}(x^1,x^2) & = (\koni(x^{1}) - \koni(x^{2}) )^-.
\end{align*}
As stated earlier, this coupling can be seen as a joint distribution for two cells, described by their protein levels, whose dynamics are given by \eqref{eq : generator P n genes}. With this coupling, the cells are not independent but rather try to synchronise by keeping their respective protein levels as similar as possible. The deterministic behaviour of our model ensures that, in the absence of jumps, the difference between two trajectories of \cref{model : Simplified Mathematical Model} decreases to 0. This difference is also maintained when there are simultaneous jumps. Indeed if gene \( i \) bursts in both cells, the burst sizes are coupled to be equal. Common jumps occur at the maximum possible rate and terms corresponding to single jumps are included to preserve marginal distributions. Note that if \( X^1(0) = X^2(0) \) almost surely, then \( X^1(t) = X^2(t) \) a.s. for all \( t \ge 0 \). By summing the jump rates in \eqref{eq : generator X1 X2 n genes}, we find that such unilateral jumps occur at a rate \( \Vert \kon(x^{1}) - \kon(x^{2}) \Vert \). Consequently, the process \( \left( \Vert X^{1}(t) - X^{2}(t) \Vert \right)_{t \ge 0} \) is not Markovian. To control this difference process and enable explicit computations, we add another one-dimensional Markovian process \( \Comp \) to the previous coupling. This results in the full coupling with infinitesimal generator
\begin{align} \label{eq : generator (X1, X2, U) n genes}
L_{\text{P-P-U}} = & \sum_{i=1}^n \left( d_{1,i} \left( \degradop_{x^1_i} + \degradop_{x^2_i} \right) + \jumpratecone{i} \burstop_{x^{1}_i,u} + \jumpratectwo{i} \burstop_{x^{2}_i,u} + \commonjumprate{i} \burstop_{x^1_i,x^2_i} \right) \nonumber \\
& + \mindun \degradop_u + \lambda_{\mathrm{U}} \burstop_u, 
\end{align}
where
\[ \lambda_{\text{U}}(x^1, x^2, u) = r (1 \wedge \lipnorm u) - \Vert \kon(x^1) - \kon(x^2) \Vert, \]
and other jump rates are the same as the previous coupling at \eqref{eq : generator X1 X2 n genes}. The next proposition shows that \( \Comp \) bounds the difference in protein levels at all time, if it controls it at time 0. 
\begin{proposition} \label{prop : maj diff P-P}
If the inequality
\begin{equation} \label{eq : maj diff P-P}
\Vert X^{1}(t) - X^{2}(t) \Vert \le \CompL(t),
\end{equation}
is verified at time \( t = 0 \), then it is verified for all \( t \ge 0 \).
\end{proposition}
To prove \cref{prop : maj diff P-P}, we use the explicit solution define at \eqref{eq_:_sol_analytique_EDO_modele_P}.
\begin{proof}[Proof of \cref{prop : maj diff P-P}]
Let's assume that inequality \eqref{eq : maj diff P-P} is verified at \( t = 0 \). First, inequality \eqref{eq : maj diff P-P} is preserved by the flow. Indeed, for \( t \ge 0 \) before the next jump, we have
\[ \Vert X^{1}(t) - X^{2}(t) \Vert \le \Vert X^{1}(0) - X^{2}(0) \Vert e^{- \underline{d}_1 t} \le \CompL(t). \]
Now, we will show that inequality \eqref{eq : maj diff P-P} is preserved when there is a jump. Without loss of generality, we can assume that the first jump occurs at time \( t = 0 \). The hypothesis must thus be specified and is now
\[ \Vert X^{1}(0^-) - X^{2}(0^-) \Vert \le \CompL(0^-),\]
where the exponent \( - \) (resp. \( + \) ) indicates the left limit (resp. the right limit) of the process. First we assume that the jump at time \( t = 0 \) is a simultaneous jump of \( (X^{1}_i(t))_{t \ge 0} \) and \( (X^{2}_i(t))_{t \ge 0} \) for some \( i \in \seg{n} \). In that case there is no jump of \( \Comp \). We thus have
\[ X^1_i(0^+) = X_i^1(0^-) + H, \quad X^2_i(0^+) = X^2_i(0^-) + H, \quad U(0^+) = U(0^-), \]
where \( H \sim \Exp(1) \) and all other coordinates being continuous. Thus, for \( t \ge 0 \) before the next jump, we have
\[
\Vert X^{1}(t) - X^{2}(t) \Vert \le \Vert X^{1}(0^-) - X^{2}(0^-) \Vert e^{- \underline{d}_1 t} \le \CompL(t).
\]
Similarly, in the case of unilateral jumps, we assume that, for some \( i \in \seg{n} \), there is a jump of \( (X^{1}_i(t))_{t \ge 0} \) only. In that case, there is also a jump of \( \Comp \). Namely
\[ X^1_i(0^+) = X^1_i(0^-) + H \quad \text{and} \quad U(0^+) = U(0^-) + H, \]
where \( H \sim \Exp(1) \) and all other coordinates being continuous. Thus, for \( t \ge 0 \) before the next jump, we have
\[ \Vert X^{1}(t) - X^{2}(t) \Vert \le \sum_{j=1}^n \vert X^{1}_j(0^-) - X^{2}_j(0^-) \vert e^{-d_{1,j} t} + H e^{- d_{1,i} t} \le \CompL(t). \]
With the same use of the triangle inequality we obtain the same result when there is a jump of \( (X^{2}_i(t))_{t \ge 0} \) only.
\end{proof}
Using the previous proposition and \cref{hyp : kon bounded and Lipschitz} we get the following result.
\begin{corollary}
If the inequality \eqref{eq : maj diff P-P} is verified at time \( t = 0 \), then for any value \( (x^{1}, x^{2}, \CompSmL) \) of the process \( (X^{1}(t), X^{2}(t), \CompL(t))_{t \ge 0} \) we have
\[ \Vert \kon (x^{1}) - \kon (x^{2}) \Vert \le r \left( 1 \wedge \lipnorm \CompSmL \right). \] 
\end{corollary}

\subsection{Study of the companion process} \label{subsection_:_etude_processus_compagnon}

In this section we show crucial properties of \( \Comp \) in order to control its expectation at each time point. 

\bigskip

The naive (yet usefull) bound on the expectation of \( \Comp \) is the following.
\begin{proposition} \label{prop_:_borne_uniforme_constante_esperance_U(t)}
For all \( t \ge 0 \), 
\[ \Esp[U(t)] \le \Esp[U(0)] \vee \roverdun, \]
where \( \roverdun \) is defined at \eqref{eq_:_proba_arret_des_sauts_processus_compagnon}.
\end{proposition}

\begin{proof} From \eqref{eq : generator (X1, X2, U) n genes} we deduce that the marginal infinitesimal generator of \( \Comp \) is
\[ L_{\mathrm{U}} = - \mindun \degradop + \lambda \burstop, \quad \text{where} \quad \lambda(\CompSmL) = r(1 \wedge \lipnorm \CompSmL). \]
Since the expectation of the companion process satisfies the following ODE
\[ \frac{\mathrm{d}}{\mathrm{d}t} \Esp[\CompL(t)] = \Esp[ L_{\mathrm{U}}\CompL(t)], \]
we have the inequality
\[\frac{\mathrm{d}}{\mathrm{d}t} \Esp[\CompL(t)] = -\mindun \Esp[\CompL(t)] + r \Esp[1 \wedge (\lipnorm \CompL(t))] \le  -\mindun \Esp[\CompL(t)] + r. \]
Using Grönwall lemma, we obtain
\[ \Esp[\CompL(t)] \le \CompL(0) e^{-\mindun t} + \roverdun (1 - e^{- \mindun t}) \le \Esp[\CompL(0)] \vee \roverdun. \]
\end{proof}

\begin{proposition} \label{prop : companion process finite number of jumps}
The process \( \Comp \) has almost surely a finite number of jumps.
\end{proposition}

\begin{proof}[Proof of \cref{prop : companion process finite number of jumps}]
Let \( T \) be the waiting time before the next jump of \( \Comp \). For \( t \ge 0 \), we have 
\begin{equation}\label{eq_:_survival_function_waiting_time}
\Prob(T > t \vert \CompL(0) = \CompSmL_0) = \exp \left( - r \int_0^t  1 \wedge (\lipnorm \CompSmL_0 e^{-\underline{d}_1 s}) \mathrm{d}s \right).
\end{equation}
For \( s \in [0,t] \) and \( \CompSmL_0 > 0 \) we have 
\[ 1 \ge \lipnorm \CompSmL_0 e^{- \underline{d}_1 s} \Longleftrightarrow s \ge \dfrac{1}{\underline{d}_1} \ln \left( \lipnorm \CompSmL_0 \right). \]
We define the time where both terms of the minimum are equal
\[ t^* = \dfrac{1}{\underline{d}_1} \ln \left( \lipnorm \CompSmL_0 \right). \]
The expression of the integral in \eqref{eq_:_survival_function_waiting_time} simplify differently according to the sign of \( t^* \). For \( t^* \le 0 \) (i.e. \( \CompSmL_0 \le 1 \slash \lipnorm \)), we have
\[
\Prob(T > t \vert \CompL(0) = \CompSmL_0) = \exp \left(- \lipnorm \roverdun \CompSmL_0 (1 - e^{- \underline{d}_1 t}) \right),
\]
and for \( t^* > 0 \) (i.e. \( \CompSmL_0 > 1 \slash \lipnorm \)),
\[
\Prob(T > t \vert \CompL(0) = \CompSmL_0) = \left\{ \begin{array}{ll}
e^{-r t} & \text{if } t \le t^* \\
\left( \lipnorm \CompSmL_0 \right)^{- \roverdun} e^{ \roverdun \left( \lipnorm \CompSmL_0 e^{- \mindun t} - 1 \right) } & \text{if } t > t^*
\end{array} \right.
\]
The probability for \( T \) to be infinite, conditionally to the initial value, is given by
\begin{equation}\label{eq_:_proba_temps_attente_infini_conditionnellement_valeur_init}
\Prob( T = + \infty \vert \CompL(0) = \CompSmL_0) = e^{- \lipnorm \roverdun \CompSmL_0} \mathds{1}_{\CompSmL_0 \le 1 \slash \lipnorm} + \left( \lipnorm \CompSmL_0 \right)^{-\roverdun} e^{- \roverdun} \mathds{1}_{\CompSmL_0 > 1 \slash \lipnorm}
\end{equation}
Note that this function is consistent for \( \CompL(0) = 0 \). The application 
\[ \CompSmL_0 \longmapsto \Prob(T = + \infty \vert \CompL(0) = \CompSmL_0) \]
is convex. Using Jensen inequality and since the obtained bound is a non increasing function of \( \mathbb{E}[\CompL(0)] \) we get
\[ \Prob(T = + \infty) \ge e^{- \roverdun (1 \wedge \lipnorm u)} \left( \frac{1}{1 \vee \lipnorm u} \right)^{\roverdun} \]
where \( u = \Esp[U(0)] \vee \roverdun \). The bound provided in \cref{prop_:_borne_uniforme_constante_esperance_U(t)} ensure that for \emph{any} inter-arrival time \( T \), the previous inequality is satisfied. This probability being strictly positive, we conclude that almost surely an inter-arrival time is going to be infinite, causing \( \Comp \) to stop jumping.
\end{proof}

We can now define the instant of the final jump of \( \Comp \) by
\[
H = \sum_{i=1}^N T_i,
\]
where \( N \) is the random variable associated to the number of jumps and \( (T_i)_{1 \le i \le N} \) are the \emph{finite} waiting times between two jumps. We now try to control the exponential moment of \( H \). First we can stochastically upper bound \( N \).

\begin{proposition} \label{prop_:_maj_sto_nbre_saut_U(t)}
The distribution of \( N \) is stochastically smaller than a geometric distribution of parameter
\[ p^* = e^{- \roverdun (1 \wedge \lipnorm u)} \left( \frac{1}{1 \vee \lipnorm u} \right)^{\roverdun} \]
where \( u = \Esp[U(0)] \vee \roverdun \).
\end{proposition}

\begin{proof}[Proof of \cref{prop_:_maj_sto_nbre_saut_U(t)}]
In the proof of \cref{prop : companion process finite number of jumps} we have shown that 
\[ p_i = \Prob( T_i = + \infty ) \ge p^* > 0,\]
for every \( i \in \seg{N} \). Then for \( k \in \N \)
\[ \Prob(N \le k) = 1 - \Prob(N \ge k + 1) = 1 - \prod_{i=1}^{k+1} (1 - p_i) \ge 1 - (1 - p^*)^{k+1}. \]
Note that we are using the following convention : if \( X \)  has a geometric distribution of parameter \( p \), for \( k \in \N \)
\[ \Prob(X = k) = p(1-p)^{k}. \]
\end{proof}

The following proposition uniformly upper bounds the distribution of finite inter-arrival times of \( \Comp \) with respect to the initial condition. Since the proof is long and is not mathematically relevant, it is moved to appendix.
\begin{proposition} \label{prop : maj sto finite waiting times}
Let \( T \) be the waiting time before the next jump of \( \Comp \). For all \( t \ge 0 \), 
\[ \inf_{\CompSmL_0 > 0} \Prob( T \le t \vert T < + \infty, \CompL(0) = \CompSmL_0) \ge 1 - e^{- \tau t}, \]
where \( \tau =  r \wedge \mindun \).
\end{proposition}

It follows from that result and the Markov property, that the distribution of any waiting time between jumps is stochastically bounded by the exponential distribution of parameter \( \tau \). Classically, this means there exists a coupling of any finite waiting time \( T_i \) with a exponential random variable \( V_i \) such that 
\[ T_i \le V_i \; \text{a.s.} \]
Similarly, \cref{prop_:_maj_sto_nbre_saut_U(t)} ensures there exists a coupling of \( N \) with a variable \( N' \) such that
\[ N' \sim \mathrm{Geometric}(p^*) \quad \text{and} \quad N \le N' \; \text{a.s.} \]
However, since both stochastic bounds do not depend of the initial condition, here we can go further and explicitly construct a coupling of all random variables \( N' \), \( V_1, \dots, V_{N'} \) such that they are mutually independent. The simultaneous construction of all variables along with the companion process is made in \cref{algo1}. This algorithm uses the inverse transformation sampling method. Since all uniform random variables used through out the algorithm are independent, the transformed variables are independent as well. Note that this algorithm is not numerically efficient since it requires to invert the function \eqref{eq_:_proba_temps_attente_infini_conditionnellement_valeur_init} at each step for a new value of \( \CompSmL_0 \). Nevertheless, it has theoretical relevance to exhibit a coupling of all variables on the same probability space. The main point of this algorithm is that it decides beforehand if the next waiting time takes infinite value, which the usual reject simulation algorithm fails to do.

\begin{algorithm}[ht]
\caption{{\small Simulation of the companion process \( \Comp \) at a fixed time point \( t \ge 0 \). This algorithm differs from the thinning algorithm classically used to simulate PDMPs. Indeed, it uses the specific behavior of \( \Comp \) by deciding beforehand if \( \Comp \) has another jump, or not. We emphasise that this algorithm is numerically inefficient since it requires to invert, every time there is a jump, a cumulative distribution function. However this algorithm shows how to simultaneously construct \( \Comp \) and all the random variables needed in \cref{proof_:_maj_expectation_companion_process}, namely \( N, N' \) and \( T_1, \dots, T_N, V_1, \dots V_{N'} \).}}
\label{algo1}
\begin{algorithmic}[1]
\Require initial state $ \CompSmL $, fixed simulation time $t > 0$ and, if \( \CompSmL \) is random, \( \overline{u} = \Esp[\CompSmL] \)
\State $\CompL \gets \CompSmL$ \Comment{\emph{Initialize companion process}}
\State \( k \gets 0 \) \Comment{\emph{Initialize jump counter}}
\State $T_{\mathrm{current}} \gets 0$ \Comment{\emph{Initialize current simulation time}}
\State \( p_{U} \gets e^{- \lipnorm \roverdun U} \mathds{1}_{U \le 1 \slash \lipnorm} + \left( \lipnorm U \right)^{-\roverdun} e^{- \roverdun} \mathds{1}_{U > 1 \slash \lipnorm}\)
\newline 
\Comment{\emph{Probability for next waiting time of \( \Comp \) to have infinite value given the initial condition~\eqref{eq_:_proba_temps_attente_infini_conditionnellement_valeur_init}}}
\While{$T_{\mathrm{current}} < t$}
\State \( U_{\mathrm{old}} \gets U \) 
\State \( T_{\mathrm{old}} \gets T_{\mathrm{current}} \)
\State \( W \gets \mathrm{Unif}([0,1]) \)
\newline
\Comment{\emph{Decide if next waiting time of \( \CompL \) has infinite value}}
\If{\( W \le p^* \)}
\Comment{\emph{\( p^* = p^*(\overline{u}) \) is defined at \eqref{eq_:_proba_arret_des_sauts_processus_compagnon}}}
\State \( N, N' \gets k \) \Comment{\emph{The companion process stops jumping and \( N' = N \)}}
\Break \Comment{\emph{Since the companion process is now deterministic}}
\EndIf
\If{\( p^*  \le W \le p_U \)} 
\State \( N \gets k \) \Comment{\emph{The companion process stops jumping but \( N \neq N' \)}}
\State \( N' \gets k + \mathrm{Geometric}(p^*) \) \Comment{\emph{Simulate \( N' \ge N \)}}
\State \( V_k, \dots, V_{N'} \gets \mathrm{Exp}(\tau) \)
\Break \Comment{\emph{Since the companion process is now deterministic}}
\EndIf
\If{\( p_{\CompL} \le W \)} \Comment{\emph{The companion process  has another jump}}
    \State  \( S \gets \mathrm{Unif}([0,1]) \) 
    \State \( T \gets F^-(S) \)
    \Comment{\emph{\( F^- \) denotes the generalized inverse of \eqref{eq_:_fct_repartition_T_conditionnelle_T_fini_et_u0}}}
    \newline
    \Comment{\emph{Draw the waiting time before the jump}}
    \State \( V \gets \tau^{-1} \ln(1 + S)\) \Comment{\emph{Draw an exponential variable of parameter \( \tau \)}}
    \State \( T_k \gets T \), \( V_k \gets V \) \Comment{\emph{Record variables. By construction \( V_k \ge T_k \)}}
    \State \( U \gets U e^{- \mindun T } + \mathrm{Exp}(1) \) \Comment{\emph{Update the value of the companion process}}
    \State \( p_{U} \gets e^{- \lipnorm \roverdun U} \mathds{1}_{U \le 1 \slash \lipnorm} + \left( \lipnorm U \right)^{-\roverdun} e^{- \roverdun} \mathds{1}_{U > 1 \slash \lipnorm}\)
    \newline
    \Comment{\emph{Update the probability for next waiting time to have an infinite value}}
\State $T_{\mathrm{current}} \gets T_{\mathrm{current}} + T$ \Comment{\emph{Update current simulation time}}
\State \( k \gets k + 1 \) \Comment{\emph{Update jump counter}}
\EndIf
\EndWhile
\State \Return \( U_{\mathrm{old}} e^{- \mindun (t - T_{\mathrm{old}})} \) 
\end{algorithmic}
\end{algorithm}
 
\begin{proposition} \label{prop : maj probability generating last jump instant}
For \( s \in [0, p^* \tau] \), the probability-generating function of \( H \) verifies
\[
\fctgeneh(s) = \Esp \left[ e^{sH} \right] \le \dfrac{\tau - s}{p^* \tau - s}. 
\]
\end{proposition}

\begin{proof}[Proof of \cref{prop : maj probability generating last jump instant}]
With the definition of \( H \) we have
\[
\fctgeneh(s) = \Esp \left[ e^{sH} \right] = \Esp \left[ \prod_{i=1}^N e^{sT_i} \right] \le \mathbb{E} \left[ \prod_{i=1}^{N'} e^{sV_i} \right],
\]
where \( N' \), \(V_1, \dots, V_{N'} \) are constructed in \cref{algo1}. Thus 
\[ G(s) \le \Esp \left[ \left( \frac{\tau}{\tau - s} \right)^{N'} \right] = \sum_{k=0}^{+ \infty} \left( \frac{\tau}{\tau - s} \right)^k \Prob( N' = k) = \frac{p^* (\tau - s)}{p^* \tau - s}. \]
\end{proof}  

\begin{proposition} \label{prop : maj expectation companion process}
For all \( t \ge 0 \) we have
\[
\mathbb{E} \left[ \CompL(t) \right] \le \left( u + \big(\mathbb{E}[\CompL(0)] + \tau (1 - p^*)\big) \frac{p^* \tau \mindun}{e (p^* \tau + \mindun)} t \right) \exp \left( - \frac{p^* \tau \mindun}{p^* \tau + \mindun} t \right),
\]
where \( p^* \) is defined at \eqref{eq_:_proba_arret_des_sauts_processus_compagnon}, \( u = \mathbb{E}[\CompL(0)] \vee \roverdun \), and \( \tau = r \wedge \mindun \).
\end{proposition}

\begin{proof}[Proof of \cref{prop : maj expectation companion process}]\label{proof_:_maj_expectation_companion_process}
Let \( \beta \in ]0,1[ \). For \( t \ge 0 \) the expectation of the companion process can be decomposed as follow
\[
\mathbb{E} \left[ \CompL(t) \right] = \mathbb{E} \left[ \CompL(t) \mathds{1}_{\{H < \beta t\}} \right] + \mathbb{E} \left[ \CompL(t) \mathds{1}_{\{H \ge \beta t\}} \right].
\]
If \( H < \beta t \), the companion process has stopped jumping and its expectation can be controlled using \cref{prop_:_borne_uniforme_constante_esperance_U(t)}
\begin{equation}\label{eq_:_majoration_esperance_1terme}
\Esp \left[ \CompL(t) \mathds{1}_{\{H < \beta t\}} \right] \le \Esp \left[ \CompL(\beta t) e^{- \mindun ( t - \beta t)} \mathds{1}_{\{H < \beta t\}} \right] \le \left( \mathbb{E}[\CompL(0)] \vee \roverdun \right) e^{- \mindun ( 1 - \beta)t}. 
\end{equation}
For any \( s \in [0, p^* \tau ] \), the second term verifies
\[ \mathbb{E} \left[ \CompL(t) \mathds{1}_{\{H \ge \beta t\}} \right] = \mathbb{E} \left[ \CompL(t) \mathds{1}_{\{H \ge \beta t\}} e^{-sH} e^{sH} \right] \le e^{-s\beta t} \mathbb{E} \left[ e^{sH} \CompL(t) \right]. 
\]
Naively, \( \CompL(t) \) is upper bounded by the sum of its initial value and all its jumps. Namely
\[ \CompL(t) \le \CompL(0) + S, \quad \text{where} \quad S = \sum_{i=1}^N E_i, \]
with \(E_1, \dots, E_N \) being the jumps of \( \Comp \).
Thus
\[ \mathbb{E}[\CompL(t) e^{sH}] \le \mathbb{E}[(\CompL(0) + S) e^{sH}]. \] 
The key is now to replace random variables in the previous equation by other random variables that are almost surely greater and mutually independent. Indeed in \cref{algo1} we constructed (simultaneously)\( N' \), \( V_1, \dots, V_{N'} \) mutually independent such that
\[ H \le \sum_{i=1}^{N'} V_i \; \text{a.s.} \]
Here we also introduce 
\[ S' = \sum_{i=1}^{N'} E_i \ge S \; \text{a.s.}, \]
where \( E_{N+1}, \dots, E_{N'} \) are independent exponential random variables of parameter 1. Here it is important to point out that \( N' \), \( S' \) and \(V_1, \dots, V_{N'} \) are independent of \( \CompL(0) \). (It results from the construction made in \cref{algo1}). Finally using the independance and \cref{prop : maj probability generating last jump instant} we get
\[ \mathbb{E}[\CompL(t)e^{sH}] \le \mathbb{E}\left[ (\CompL(0) + S') \prod_{i=1}^{N'}e^{sV_i} \right] \le \mathbb{E}[\CompL(0)] \frac{p^* (\tau - s)}{p^* \tau - s} + \mathbb{E} \left[ S' \prod_{i=1}^{N'} e^{sV_i} \right]. \]
It remains to explicit
\[ \mathbb{E} \left[ S' \prod_{i=1}^{N'} e^{sV_i} \right] = \mathbb{E} \left[ N' \left( \frac{\tau}{\tau - s} \right)^{N'} \right], \]
since conditionally to \( N' \), \( S' \) follows an Erlang distribution of parameter \( (N', 1) \). We have
\[ \mathbb{E} \left[ N' \left( \frac{\tau}{\tau - s} \right)^{N'} \right] =  p^* \sum_{k=1}^{+ \infty} k \left( \frac{(1 - p^*)\tau}{\tau - s} \right)^k = p^* (1 - p^*) \tau \frac{\tau - s}{p^* \tau - s}, \]
and finally
\[ \mathbb{E}[\CompL(t) e^{sH}] \le \frac{p^* (\tau - s)}{p^* \tau - s} \left( \mathbb{E}[\CompL(0)] + \tau (1 - p^*) \right). \]
Since for all \( 0 < s < p^* \tau \)
\[ e^{-s \beta t} \dfrac{\tau - s}{p^* \tau - s} \le e^{-s \beta t} \left( \dfrac{\tau}{p^* \tau -s} \right), \]
and because the application 
\[ h(s) = \ln \left( \dfrac{\tau}{p^* \tau - s} \right) - s \beta t, \]
reachs its minimum for \( s = p^* \tau - (\beta t)^{-1} \), we can optimise the bound and obtain
\begin{equation} \label{eq_:_majoration_esperance_2terme}
\mathbb{E} \left[ \CompL(t) \mathds{1}_{\{H \ge \beta t\}} \right] \le p^* \tau \beta t e \left( \mathbb{E}[\CompL(0)] + \tau (1-p^*) \right) e^{- p^* \tau \beta t}.
\end{equation}
Adding both \eqref{eq_:_majoration_esperance_1terme} and \eqref{eq_:_majoration_esperance_2terme}, we get
\[
\mathbb{E} \left[ \CompL(t) \right] \le (\mathbb{E}[\CompL(0)] \vee \roverdun) e^{- \mindun (1- \beta) t} + p^* \tau \beta t e (\mathbb{E}[\CompL(0)] + \tau (1 - p^*)) e^{- p^* \tau \beta t - 1}.
\]
Finallly set \( \beta = \frac{\mindun}{p^* \tau + \mindun} \) to obtain the maximal convergence rate which factorize both exponential
\[
\mathbb{E} \left[ \CompL(t) \right] \le \left( u + \frac{p^* \tau \mindun}{e (p^* \tau + \mindun)} (\mathbb{E}[\CompL(0)] + \tau (1 - p^*)) t \right) \exp \left( - \frac{p^* \tau \mindun}{p^* \tau + \mindun} t \right),
\]
where \( u = \mathbb{E}[\CompL(0)] \vee \roverdun \).
\end{proof}

\subsection{The companion process in the case of the MP model} \label{subsection_:_processus_compagnon_modele_MP}

In this section, we construct, as in \cref{subsection_:_processus_compagnon_P_modele}, a coupling 
\[ (Y^{1}(t), Z^{1}(t), Y^{2}(t), Z^{2}(t), \CompL(t))_{t \ge 0} \in (\R_+)^n \times (\R_+)^n \times (\R_+)^n \times (\R_+)^n \times \R_+ \]
of two copies of \cref{model : Mathematical Bursty Model} and the companion process. The resulting coupling is also a Markov process that rests on the same idea as the one at \eqref{eq : generator (X1, X2, U) n genes}.

\begin{align} 
L_{\text{MP-MP-U}}f = & \sum_{i=1}^n \bigg( d_{0,i} \left( \degradop_{y_i^1} + \degradop_{y_i^2} \right) + d_{1,i} \left( \creatop_{z_i^1}^{y_i^1} + \degradop_{z_i^1} + \creatop_{z_i^2}^{y_i^2} + \degradop_{z_i^2} \right) \nonumber  \\
& + \commonjumprate{i} \burstop_{y_i^1, y_i^2}^{\rho_i, \rho_i} + \jumpratecone{i} \burstop_{y_i^{1}, u}^{\rho_i} + \jumpratectwo{i} \burstop_{y_i^{2}, u}^{\rho_i} \bigg) \nonumber \\
& + \mindun \degradop_u + \lambda_{\text{U}} \burstop_u, \label{eq_:_generator_MP_MP_U}
\end{align}
where
\begin{align*}
\commonjumprate{i}(y^1, z^1, y^2, z^2, u) & = \koni(z^1) \wedge \koni(z^2), \\
\jumpratecone{i}(y^1, z^1, y^2, z^2, u) & = (\koni(z^1) - \koni(z^2))^+, \\
\jumpratectwo{i}(y^1, z^1, y^2, z^2, u) & = (\koni(z^1) - \koni(z^2))^-, \\
\lambda_{\text{U}}(y^1, z^1, y^2, z^2, u) & = r \left( 1 \wedge \lipnorm \CompSmL \right), \\
\rho_i & = \epsi{i}^{-1}.
\end{align*}

\bigskip

In this case, in order to bound the difference between proteins at all times, the companion process must initially verify a slightly stronger property.
\begin{proposition} \label{prop : maj diff MP-MP}
If the inequality 
\begin{equation} \label{eq : maj diff MP-MP}
\Vert Z^{1}(t) - Z^{2}(t) \Vert +  \weightnorm{Y^{1}(t) - Y^{2}(t)}  \le \CompL(t) 
\end{equation}
is verified at time \( t = 0 \), then it is verified for all \( t \ge 0 \).
\end{proposition}

\begin{proof}[Proof of \cref{prop : maj diff MP-MP}]
Assuming \cref{eq : maj diff MP-MP} is verified for \( t = 0 \), then for \( t \ge 0 \) before the next jump we have
\begin{align*}
\Vert Z^{1}(t) - Z^{2}(t) \Vert \le & \sum_{j=1}^n \vert Z_j^{1}(0) - Z_j^{2}(0) \vert e^{-d_{1,j} t} \\
& + \epsi{j} \vert Y_j^{1}(0) - Y_j^{2}(0)) \vert ( e^{-d_{1,j} t} - e^{-d_{0,j} t} )
\end{align*}
and
\[
\weightnorm{ Y^{1}(t) - Y^{2}(t) } = \sum_{j=1}^n \epsi{j} \vert Y_j^{1}(0) - Y_j^{2}(0) \vert e^{-d_{0,j} t}.
\]
By adding both terms, we get
\begin{align*}
\Vert Z^{1}(t) - Z^{2}(t) \Vert + & \weightnorm{ Y^{1}(t) - Y^{2}(t) } \\ 
\le & e^{ - \mindun t} \sum_{i=1}^n \Big( \vert Z_j^{1}(0) - Z_j^{2}(0) \vert + \epsi{j} \vert Y_j^{1}(0) - Y_j^{2}(0) \vert \Big) \\
\le & \CompL(t).
\end{align*}
This proves that the bound is preserved through deterministic phases. We will now show that the bound is preserved when a jump occurs. We thus need to precise that the hypothesis is now 
\[ \Vert Z^{1}(0) - Z^{2}(0) \Vert + \weightnorm{Y^{1}(0^-) - Y^{2}(0^-)} \le \CompL(0^-). \] 
Note that \( (Z(t))_{t \ge 0} \) has continuous trajectories so, 
\[Z(t) = Z(t^+) = Z(t^-), \]
for all \( t \ge 0 \). This is why, for this process, we will keep the notation \( Z(t) \). First, we consider the event of simultaneous jumps, which can be assumed to occur at time 0 without loss of generality. That is, for some \( i \in \seg{n} \), 
\[ Y^1_i(0^+) = Y^1_i(0^-) +  \epsi{i}^{-1} H , \quad Y^2_i(0^+) = Y^2_i(0^-) + \epsi{i}^{-1} H \quad \text{where} \quad H \sim \Exp(1). \]
In the case of such an event, the other coordinates will remain the same and the companion process won't jump. Namely for every \( j \neq i \), 
\[ Y^1_j(0^+) = Y^1_j(0^-), \quad Y^2_j(0^+) = Y^2_j(0^-), \quad \text{and} \quad \CompL(0^+) = \CompL(0^-). \]
Thus, for every \( t \ge 0 \) before the next jump, we have
\begin{align*}
\Vert Z^1(t) - Z^2(t) \Vert +  \weightnorm{Y^1(t) & - Y^2(t)} \\
& \le e^{- \mindun t} \left( \Vert Z^1(0) - Z^2(0) \Vert + \weightnorm{ Y^1(0^-) - Y^2(0^-)} \right) \\
& \le \CompL(t).
\end{align*} 
Similarly, in the case of unilateral jumps, we assume that for some \( i \in \seg{n} \),
\[ Y_i^{1}(0^+) = Y_i^1(0^-) +  \epsi{i}^{-1} H \quad \text{where} \quad H \sim \Exp(1). \]
In that case, 
\[ \CompL(0^+) = \CompL(0^-) + H, \]
and all other coordinates remain the same.
For \( t \ge 0 \) until the next jump we have
\begin{align*}
\Vert Z^{1}(t) - Z^{2}(t) \Vert = & \left( \sum_{j \neq i} \vert Z^{1}_j(t) - Z^2_j(t) \vert \right) + \Big| (Z_i^{1}(0) - Z_i^{2}(0)) e^{-d_{1,i} t} \\
& + \epsi{i} (Y_i^{1}(0^-) + \epsi{i}^{-1} H - Y_i^{2}(0^-))( e^{-d_{1,i} t} - e^{-d_{0,i} t}) \Big| \\
\le & \left( \sum_{j \neq i}  \vert Z^{1}_j(t) - Z^2(t) \vert \right) + \vert (Z_i^{1}(0) - Z_i^{2}(0)) \vert e^{-d_{1,i} t} \\
& + \epsi{i} \vert Y_i^{1}(0^-) - Y_i^{2}(0^-) \vert ( e^{-d_{1,i} t} - e^{-d_{0,i} t} ) \\
& + H (e^{d_{1,i} t} - e^{-d_{0,i} t}) 
\end{align*} 
and
\begin{align*}
\weightnorm{ Y^{1}(t) - Y^{2}(t) } \le  \sum_{j=1}^n \left( \epsi{j} \vert Y_j^{1}(0^-) - Y_j^{2}(0^-) \vert e^{-d_{0,j} t} \right) + H e^{-d_{0,i} t}.
\end{align*}
Adding the two terms together gives us
\begin{align*}
\Vert Z^{1}(t) - & Z^{2}(t)  \Vert + \weightnorm{ Y^{1}(t) - Y^{2}(t) } \\
& \le \sum_{j=1}^n \Big( \vert Z_j^{1}(0) - Z_j^{2}(0) \vert e^{-d_{1,j} t} + \epsi{j} \vert Y_j^{1}(0^-) - Y_j^{2}(0^-) \vert e^{-d_{1,j} t} \Big) + H e^{-d_{1,i} t} \\
& \le \CompL(t).
\end{align*}
With the same use of the triangle inequality, we show that the bound is also preserved when only a coordinate of \( (Y^2(t))_{t \ge 0} \) is impacted by a jump. 
\end{proof}

\subsection{Proof of final results}\label{subsection_:_preuve_resultat_final_quantitative_ergodicity}

In this section, we prove \cref{thm : bound over wasserstein distance P-P} and \cref{thm : bound over wasserstein distance MP-MP}. These proofs are simply the concatenation of the Wasserstein metric definition and results obtained on the companion process in \cref{subsection_:_etude_processus_compagnon}. 

\begin{proof}[Proof of \cref{thm : bound over wasserstein distance P-P}]
We consider the coupling \( (X^1(t), X^2(t), \CompL(t))_{t \ge 0} \) with infinitesimal generator defined at \eqref{eq : generator (X1, X2, U) n genes}. We set
\[ \CompL(0) = \Vert X^1(0) - X^2(0) \Vert. \]
According to \cref{prop : maj diff P-P} 
\[ \Vert X^1(t) - X^2(t) \Vert \le \CompL(t), \]
for \( t \ge 0 \). Thus, with notation of \cref{def_:_notation_measures}, we have
\[ W_1(\mu^1(t), \mu^2(t)) \le \Esp\left[ \Vert X^1(t) - X^2(t) \Vert \right] \le \Esp[\CompL(t)], \]
since the Wasserstein metric is defined as an infimum. With \cref{prop : maj expectation companion process} we have
\begin{equation}\label{eq_:_preuve_thm_2.2}
\Esp[\CompL(t)] \le \left( \Esp[\CompL(0)] \vee \roverdun + \big(\mathbb{E}[\CompL(0)] + \tau (1 - p^*) \big) e^{-1} \gamma t \right) \exp \left( - \gamma t \right),
\end{equation}
where 
\[ \gamma = \frac{p^* \tau \mindun}{p^* \tau + \mindun}\quad \text{and} \quad p^* \text{ is defined in \cref{prop_:_maj_sto_nbre_saut_U(t)}}. \] 
Since \( \gamma \) is a non-decreasing function of \( p^* \) and \( p^* \) is itself a non-increasing function of \( \Esp[\CompL(0)] \), to optimise \( \gamma \) we choose 
\[ \Esp[\CompL(0)] = W_1(\mu^1(0), \mu^2(0)). \]
Note that this choice also optimise the coefficients of the polynom in \eqref{eq_:_preuve_thm_2.2}.
\end{proof}

Similarly we prove the result for \cref{model : Mathematical Bursty Model}.
\begin{proof}[Proof of \cref{thm : bound over wasserstein distance MP-MP}]
Consider the coupling \( (Y^1(t), Z^1(t), Y^2(t), Z^2(t))_{t \ge 0} \) with infinitesimal generator \eqref{eq_:_generator_MP_MP_U}. We set
\[ \CompL(0) = \Vert Z^1(0) - Z^2(0) \Vert + \weightnorm{ Y^1(0) - Y^2(0) }. \]
As in the previous proof, with notation in \cref{def_:_notation_measures}, we get 
\[ W_1(\nu^1(t), \nu^2(t)) \le \left( \Esp[\CompL(0)] \vee \roverdun + \big(\mathbb{E}[\CompL(0)] + \tau (1 - p^*)\big) e^{-1} \gamma t \right) \exp \left( - \gamma t \right), \]
with 
\[ \gamma = \frac{p^* \tau \mindun}{p^* \tau + \mindun} \]
and the optimal choice
\[ \Esp[\CompL(0)] = W_1(\nu^1(0), \nu^2(0)) + \widetilde{W}_1(\eta^1(0), \eta^2(0)), \]
where \( \widetilde{W}_1 \) stands for the Wasserstein distance on \( (\R^n, \weightnorm{\cdot}) \).
\end{proof}

\clearpage

\phantomsection
\section*{Declarations}
\addcontentsline{toc}{section}{Declarations}

\phantomsection
\addcontentsline{toc}{subsection}{Code Availability}
\paragraph{Code Availability.}\hspace{-4mm}
An implementation of the coupling algorithm in Python is available at \url{https://github.com/ulysseherbach/grn-couplings} along with examples.

\phantomsection
\addcontentsline{toc}{subsection}{Acknowledgements}
\paragraph{Acknowledgements.}\hspace{-4mm}
The second author is very grateful to Asmaa Labtaina for preliminary work on this subject.

\phantomsection
\addcontentsline{toc}{subsection}{Funding}
\paragraph{Funding.}\hspace{-4mm}
This work was supported in part by the Programmes et Équipements Prioritaires de Recherche (PEPR) Santé Numérique under Project 22-PESN-0002.

\appendix

\section{Appendices}

\begin{proof}[Proof of \cref{prop : maj sto finite waiting times}]
First note that if \( \CompSmL_0 = 0 \) then \( T = + \infty \) almost surely. Set \( t \ge 0 \), for \( \CompSmL_0 > 0 \) we have 
\[ \Prob(T \le t \vert T < + \infty, \CompL(0) = \CompSmL_0) = \frac{\Prob(T \le t \vert \CompL(0) = \CompSmL_0)}{1 - \Prob(T = + \infty \vert \CompL(0) = \CompSmL_0)}. \]
Recall that
\[ \Prob( T > t \vert \CompL(0) = \CompSmL_0) = \exp \left( - r \int_0^t 1 \wedge \lipnorm \CompSmL_0 e^{- \mindun s} \mathrm{d}s \right) \quad \text{and} \quad t^* = \frac{1}{\mindun} \ln (\lipnorm \CompSmL_0). \]
The expression of the function is fully explicit but depends of the expression of the integral and can be written conditionally to the value of \( \CompSmL_0 \) and \( t^* \) (which is a function of \( \CompSmL_0 \)). We have
\begin{align}
\Prob(T \le t \vert T < + \infty, \CompL(0) = \CompSmL_0) = & \frac{1 - e^{- \lipnorm \roverdun \CompSmL_0(1 - e^{- \mindun t})}}{1 - e^{- \lipnorm \roverdun \CompSmL_0}} \mathds{1}_{\{ \CompSmL_0 \le 1 \slash \lipnorm\}} \nonumber \\
& + \frac{1 - e^{-rt^*} e^{- \roverdun (1 - e^{- \mindun(t - t^*)})}}{1 - e^{-rt^*} e^{-\roverdun}} \mathds{1}_{\{ \CompSmL_0 > 1 \slash \lipnorm \}} \mathds{1}_{\{ t > t^* \}} \nonumber\\
& + \frac{1 - e^{-rt}}{1 - e^{-rt^*} e^{- \roverdun}} \mathds{1}_{\{ \CompSmL_0 > 1 \slash \lipnorm \}} \mathds{1}_{\{ t \le t^* \}}
\label{eq_:_fct_repartition_T_conditionnelle_T_fini_et_u0}
\end{align}

\bigskip

First consider the case \( \CompSmL_0 < 1 \slash \lipnorm \). For any positive constant \( C > 0\) and \( t \ge 0 \), the application 
\[ t \longmapsto \frac{e^{Ce^{-\mindun t}} - 1}{e^C - 1} - e^{- \mindun t}, \] 
is non-positive. Thus 
\[ 1 - \frac{1 - e^{- \lipnorm \roverdun \CompSmL_0(1 - e^{- \mindun t})}}{1 - e^{- \lipnorm \roverdun \CompSmL_0}} \ge 1 - e^{- \mindun t} \ge 1 - e^{- \tau t}. \]

\bigskip

When \( \CompSmL_0 < 1 \slash \lipnorm \) and \( t \le t^* \), we easily have
\[ \frac{1 - e^{-rt}}{1 - e^{-rt^*} e^{- \roverdun}} \ge 1 - e^{-rt} \ge 1 - e^{- \tau t}. \]

\bigskip

The last case requires more calculus but the idea is the same as before : when the jump rate is constant, the waiting time is stochastically bounded by the exponential distribution of parameter \( r \) and when it is in the decreasing phase, it is bounded by the exponential distribution of parameter \( \mindun \). Thus here it is natural to show that the waiting time a stochastically bounded by a mixture of those two distributions with density
\[ g(s) = A^{-1} \left( re^{-rs} \mathds{1}_{[0,t^*]}(s) + C \mindun e^{- \mindun s} \mathds{1}_{[t^*, + \infty[}(s) \right) , \]
where
\[ C = \roverdun e^{t^* (\mindun - r)} \quad \text{and} \quad A = 1 - e^{-r t^*}(1 - \roverdun), \]
are respectively the constant such that \( g \) continuous and the normalizing constant. For \( t \ge t^* \) the cumulative distribution function of this distribution is given by
\[ G(t) = A^{-1} \left( (1 - e^{-r t^*}) + \roverdun e^{-r t^*} \left(1 - e^{- \mindun(t - t^*)} \right) \right) = 1 - \frac{\roverdun e^{t^* ( \mindun - r)}}{1 - e^{-r t^*}(1 - \roverdun)} e^{- \mindun t}. \]
To show the stochastic bound, we study the sign of the application
\[ f : t \longmapsto \frac{1 - e^{-rt^*} e^{- \roverdun (1 - e^{- \mindun(t - t^*)})}}{1 - e^{-rt^*} e^{-\roverdun}} - G(t), \]
for all \( t \ge t^* \). Since \( f(t^*) \ge 0 \) and \( f \) tends to zero as \( t \) goes to infinity. Furthermore, the variations of \( f \)  show that it is non-negative. Indeed
\[ f' : t \mapsto \roverdun \mindun e^{- \mindun (t - t^*)} e^{- r t^* } \left( \frac{e^{- \roverdun \left(1 - e^{- \mindun (t - t^*)}\right)}}{1 - e^{-rt^* - \roverdun}} - \frac{1}{1 - e^{-rt^* (1 - \roverdun)}}\right), \]
and only the last term can change the sign of \( f' \). Since this last term is non-negative when \( t = 0 \), tends to a non-positive value when \( t \rightarrow + \infty \) and non increasing, we obtain the result on \( f \) and 
\[ f(t) \ge G(t) \quad \text{for all } t \ge t^*. \]
However this bound still depends of \( \CompSmL_0 \) through \( t^* \) in the expression of \( G \).
Since the application 
\[ s  \longmapsto \frac{\roverdun e^{s (\mindun - r)}}{1 - e^{-rs}(1 - \roverdun)}, \]
is non-increasing on \( \R_+ \), it is greater than its value at \( s = t^* \) and we have
\[ G(t) \ge 1 - \frac{\roverdun e^{-rt}}{1 - e^{-rt}(1 - \roverdun)}. \]
Finally 
\[ \frac{\roverdun e^{-rt}}{1 - e^{-rt}(1 - \roverdun)} \le e^{- \tau t}, \]
since the application 
\[ h : t \longmapsto e^{(r - \tau)t} \left( 1 - e^{-rt}(1 - \roverdun) \right) - \roverdun, \]
is non-positive. As before, we have \( h(0) = 0 \), \( h \) tends to 0 as \( t \rightarrow + \infty \) and the function \( h \) is first non-decreasing and then non-increasing. 
\end{proof}

\clearpage

\phantomsection
\addcontentsline{toc}{section}{References}
\small\bibliography{bibliography}

\end{document}